\documentclass[preprint,12pt]{article}
\usepackage[ansinew]{inputenc}
\usepackage{epsfig}
\usepackage[usenames]{color}
\usepackage{amsfonts}
\usepackage{latexsym}
\usepackage{graphicx}
\usepackage{amsmath}
\usepackage{amssymb}
\newtheorem{definition}{Definition}[section]

\newtheorem{proposition}[definition]{Proposition}

\newtheorem{example}[definition]{Example}

\newtheorem{remark}[definition]{Remark}
\newtheorem{theorem}[definition]{Theorem}
\newtheorem{e.g.}[definition]{Example}

\def\r#1{\mathbb R^{#1}}

\def\X{\mathfrak X}

\def\ovgam{\overline\gamma}

\usepackage{amssymb}

\title{The Taylor Expansion of the Exponential Map and Geometric Applications}

\author{M. G. Monera, A. Montesinos-Amilibia$^*$ and\\
 E. Sanabria-Codesal\thanks{Work partially supported by DGCYT grant no. MTM2009-08933.}}
\begin{document}
\maketitle

\begin{abstract}
In this work we consider the Taylor expansion of the exponential map of a submanifold immersed in $\mathbb{R}^n$
up to order three, in order to introduce the concepts of \textit{lateral and frontal deviation}.
We compute the directions of extreme lateral and frontal deviation for surfaces in
$\mathbb{R}^3.$ Also we compute, by using the Taylor expansion, the
directions of high contact with hyperspheres of a surface immersed
in $\mathbb{R}^4$ and the asymptotic directions of a surface
immersed in $\mathbb{R}^5.$

\end{abstract}

\emph{Keywords:} Exponential map, Surfaces,  Extremal directions,  Contact,  Normal torsion
\\
\emph{MSC:} 53A05, 53B20

\section{Introduction}
In this paper, we analyze the Taylor expansion of the exponential map up to order three of a submanifold $M$ immersed in $\r n.$ Our main goal is to show its
usefulness for the description of special contacts of the submanifolds with geometrical models. Classically, the study of the contact with hyperplanes and hyperspheres has been realized by using the family of height and squared distance functions (\cite{P1},\cite{Montesis}). As we analyze the contacts of high order, the complexity of the calculations increases. In this work, through the Taylor expansion of the exponential map, we characterize the geometry of order higher than $3$ in terms of invariants of the immersion, so that the operations be more affordable. Also this new technic give us new geometric concepts.

On the one hand, we gain some geometrical insights, as we explain now. Let $M$ be a regular surface immersed in $\r n$ and $\gamma:I \to \r{}$ be the geodesic defined in $M$ by the initial condition $\gamma'(0)= v \in T_mM.$ Let  $g: I \to  \r n$ be  the geodesic defined in $\r n$ with the initial velocity, that is $g(t)= m + v t.$ The difference $\gamma-g$ gives the geodesic deviation of the immersion for the initial condition $v.$ The Taylor  expansion of $\gamma(t) -g (t)$ begins with the second order term which is proportional to the second fundamental form of $M$ at $m$ acting upon $v,$ say $\alpha(v,v).$ It is orthogonal to $T_mM$ and  its meaning is well known. The third term has in general non-vanishing orthogonal and tangential components with respect to $T_mM.$  The orthogonal component depends essentially on the third order geometry of the surface, that is on the covariant derivative of the second fundamental form. The tangential component, on its part, depends only on the second fundamental form at $m$ and may be decomposed naturally into two components, one  tangent to $v$ and the other orthogonal to it. We call the first, the \emph{frontal deviation}, and the second, the \emph{lateral deviation}. We shall distinguish the directions $v\in T_mM$ on which the norm or the frontal deviation (resp. the lateral deviation) are extremal. In the case of $M$ being a surface, there  are in general at most four
directions of each or these classes. We shall show that the directions where the lateral deviation vanishes are the directions of higher contact of a geodesic with the normal section of the surface.

On the other hand, we obtain an expression for the normal torsion in terms of invariants related to the second fundamental form and its covariant derivative.

Finally, we compute by using the Taylor expansion of the exponential map, the directions of higher contact with hyperspheres of a surface in $\r 4,$ defined by
J. Montaldi in \cite{M}, and characterize the centers of these hyperspheres through the normal curvature and normal torsion. We also characterize the asymptotic
directions of a surface in $\r 5.$ In both cases, the results are given in terms of invariants of the immersion, so that the numerical or symbolic computation of
those directions becomes affordable, not hampered by the recourse to Monge's, isothermal or other special coordinates as in other works (\cite{M}, \cite{MRR}).

\section{Preliminaries}

Let $M$ be a differentiable manifold immersed in $\r n.$ Since all of our study will be local, we gain in brevity by assuming that it is a regular submanifold. For
each $m\in M$ we  consider the decomposition $T_m\r n= T_m M \oplus N_m M,$ where $N_mM$ denotes the normal subspace to $M$ at $m.$ Given $X\in T_m \r n ,$ that
decomposition will be written as $X=X^\top+X^\bot$ where $X^\top \in T_mM,$ $X^\bot\in N_mM.$

Let $\pi:TM\to M$ and $\pi_N:NM\to M$ denote the tangent and normal bundles respectively. If $E$ is the total space of a smooth bundle we will denote by $\Gamma(E)$
the space of its smooth sections. For the particular case of $TM$ we will put $\X(M)$ instead. We define the connection $\nabla^\top$ for $\pi$ by
$\nabla_{X}^\top Y=(D_X Y)^\top,\, X,Y\in\X(M),$ where $D$ is the Riemannian connection in $\r n$ which coincides with the directional derivative. For $\pi_N$ we
define the connection $ \nabla^\bot$ by $\nabla_{X}^\bot u=(D_X u)^\bot, \; u\in \Gamma(NM).$ These connections define a new connection $\nabla$  in
$\Gamma(T^{(r,s)}M\otimes N^{(p,q)}M)$ such that if, for example, we have $w=u\otimes Y\otimes \beta,$ where  $u\in \Gamma(NM),$ $Y\in \mathfrak{X}(M),$
$\beta\in\Gamma(T^\ast M)$ then:
  $$
\nabla_X w= (\nabla_{X}^\bot u)\otimes Y\otimes \beta + u\otimes \nabla_{X}^\top Y\otimes \beta + u\otimes Y\otimes \nabla_{X}^\top
\beta.
  $$
This connection preserves the inner product.

The second fundamental form $\alpha:\X(M)\times\X(M) \to  \Gamma(NM)$ is
the bilinear symmetric map defined by $\alpha(X,Y)=(D_XY)^\bot.$ Thus, if $u\in \Gamma(NM),$ we will have
 $u\cdot\alpha(X,Y)= -(D_Xu)\cdot Y.$

\subsection{Surfaces}

Let $M$ be a surface immersed in $\r n$ and we consider $(t_1,t_2)$ a local orthonormal frame of $TM$ on $U\subset M.$
For each $m\in U,$ the unit circle $S^1(T_mM)$ of $T_mM$ can be parameterized by the angle $\theta \in [0, 2 \pi]$ with respect to the value of $t_1
$ at $m$ and we define the map $\eta_m:S^1(T_mM)\to N_mM$ by $\eta_m(\theta)=\alpha_m(t(\theta),t(\theta)),$ where $t(\theta)=t_{1}\cos \theta + t_{2}\sin\theta.$
Therefore:
\begin{align*}
  \eta(\theta) & =   \alpha(t_1,t_1)\cos^2\theta+ \alpha(t_2,t_2)\sin^2\theta+ 2\alpha(t_1,t_2)\sin\theta\cos\theta.
\end{align*}

Putting $b_1=\alpha(t_1,t_1), $ $b_2=\alpha(t_2,t_2) $ and $b_3=\alpha(t_1,t_2),$ then
  $$
\eta(\theta)= H+B\cos2\theta+ C\sin2\theta,
  $$
where $H=\dfrac{1}{2}(b_1+ b_2),$ $B=\dfrac{1}{2}(b_1-b_2)$ and $C=b_3.$

Consider the affine subspace of $\r n$ which passes by $m$ and
is generated by $t(\theta)\in T_mM$ and $N_mM.$ The intersection of this subspace with M is a
curve that passes by $m,$ called the normal section of M determined by $t(\theta),$ and the curvature vector of this curve
coincide with $\eta_m(\theta)=\alpha_m(t(\theta),t(\theta)).$

The image of the map $\eta_m$ is an ellipse in $N_mM$ called \it curvature ellipse, \rm whose center is the vector $H_m$. Hence, this vector, called the \it mean
curvature vector, \rm does not depend on the choice of the orthonormal frame $(t_1, t_2).$  It is possible to choose this frame in such a way that $B$ and $C$
coincide with the  half-axes of the ellipse, i.e. $|B|\ge|C|$ and $B\cdot C=0.$

When the curvature ellipse at $m$  degenerates to a segment we say that the point
$m$ is \it semiumbilic \rm and if in addition a straight
line containing that segment passes by the origin then $m$ is called an \it inflection point. \rm If $m$ is semiumbilic, then the orthonormal frame
$(t_1, t_2)$ can be chosen in such a way that $C_m=0.$


\subsection{Contact theory}

Let $M_i, N_i,\, i=1,2$ be submanifolds of $I\!\! R^n$ with ${\rm dim}\,
M_1={\rm dim}\, M_2$ and ${\rm dim}\, N_1={\rm dim}\, N_2.$ We say
that the \textit{contact of} $M_1$ and $N_1$ at $y_1$ is of the same
type as the {\it contact of} $M_2$ and $N_2$ at $y_2$ if there is a
diffeomorphism germ $\Phi :( I\!\! R^n,y_1)\rightarrow (I\!\! R^n,y_2)$ such
that $\Phi (M_1)=M_2$ and $\Phi (N_1)=N_2.$ In this case we write $
K(M_1,N_1;y_1)=K(M_2,N_2;y_2). $ J. A. Montaldi gives in \cite{montaldi} the following
characterization of the notion of contact by using the terminology
of singularity theory:

\begin{theorem} Let $M_i, N_i$ $(i=1,2)$ be submanifolds of
$I\!\! R^n$ with ${\rm dim}\, M_1={\rm dim}\, M_2$ and ${\rm dim}\,
N_1={\rm dim}\, N_2.$ Let $f_i:(M_i,x_i)\rightarrow (I\!\! R^n, y_i)$
be immersion germs and $g_i:(I\!\! R^n,y_i)\rightarrow (I\!\! R^r,0)$ be
submersion germs with $(N_i,y_i)=(g_i^{-1}(0),y_i).$ In this case
$K(M_1,N_1;y_1)$ $=$ $K(M_2,N_2;y_2)$ if and only if the germ $(g_1\circ f_1, x_1)$ is
${\cal K}$-equivalent to the germ $(g_2\circ f_2, x_2)$.
\end{theorem}

Therefore, given two submanifolds $M$ and $N$ of $I\!\! R^n$, with a common point $y$,  an
immersion germ $f:(M,x)\rightarrow (I\!\! R^n, y)$ and a submersion
germ $g:( I\!\! R^n,y)\rightarrow (I\!\! R^r,0)$, such that $N= g^{-1}(0)$,
the contact of $M \equiv f(M)$ and $N$ at $y$ is completely
determined by the ${\cal K}$-singularity type of the germ $(g\circ
f, x)$ (see \cite{GG} for details on ${\cal K}$-equivalence).


When $N$ is a hypersurface, we have $r=1$, and the function germ
$(g\circ f, x)$ has a degenerate singularity if and only if its
Hessian, ${\cal H}(g\circ f)( x)$, is a degenerate quadratic form.
In such case, the tangent directions lying in the kernel of this
quadratic form are called \textit{contact directions} for $M$ and
$N$ at $y$.

We shall apply this theory to the contacts of surfaces with
hyperplanes and hyperspheres in $ \r n$. In the following $\phi:U\subset \r2\to \r n$ will be an immersed surface, where $M=\phi(U).$

\begin{definition}
The family of \it height functions \rm on $M,\;\;\lambda:  M\times S^{n-1}\to
 \r{}$  is defined as $\lambda_{u}(m) = \lambda(m,u)=m\cdot u,$ $u \in S^{n-1}$ where $S^{n-1}$ is the unit sphere in $ \r n$ centered at the origin.
\end{definition}

Varying $u$ we obtain a family of functions $\lambda_u$ on $M$ that describes all the possible contacts of $M$
with the hyperplanes on $\r n$ (\cite{M-R-R2}, \cite{MRR}).
The function $ \lambda_{u}$ has a singularity at $m=\phi(x_0,y_0) \in M$ if and only if
  $$
d_m\lambda_{u} =\left(\displaystyle{\frac{\partial
\phi}{\partial x}(x_0,y_0) \cdot u},\, \displaystyle{\frac{\partial
\phi}{\partial y}(x_0,y_0) \cdot u}\right)=(0,0),
  $$
which is equivalent to say that $u\in N_mM.$

Let $D\lambda : M \times S^{n-1} \rightarrow S^{n-1} \times I\!\! R$ be the
unfolding associated to the family $\lambda$. The singular set of $D\lambda$,
given by

$$
\Sigma (D\lambda ) = \{(m,u)\in M \times S^{n-1}:  d_m\phi\cdot u  =0\}
$$
can clearly be identified with a canal hypersurface, $CM$, of $M$ in $I\!\!R^n$. Moreover, the restriction of the
natural projection $\pi : M\times S^{n-1} \rightarrow S^{n-1}$ to the
submanifold $\Sigma (D\lambda)\equiv CM$ can be viewed as the normal Gauss
map, $\Gamma : CM \rightarrow S^{n-1}$  on the hypersurface $CM$. It is not
difficult to verify that $x$ is a degenerate singularity of $\lambda_u$ if and only if $(m,u)$ is a singular point of
$\Gamma$ if and only if ${\cal K}(m,u)=0$, where ${\cal K}$ denotes the gaussian
curvature function on $CM,$ i.e. ${\cal K}= det (d\Gamma)$, where $d\Gamma:T(CM)\rightarrow  TS^{k+n-1}.$

\begin{definition}\label{1.1}
If $m$ is a degenerate singularity (non Morse) of $ \lambda_{u},$ we say that $u$ defines a \textbf{binormal} direction
for $M$ at $m.$ The vector $v \in T_mM$ is an \textbf{asymptotic direction} at $m$ if and only if $v$ lies in the kernel of the
hessian of some height function $\lambda_{u}$ at $m.$ In this case we say that $v$ is an asymptotic direction associated to the
binormal direction $u$ at $m.$
\end{definition}

These directions were introduced in \cite{M-R-R2}, where their existence and
distribution over the generic submanifolds was analyzed.

\begin{definition}
The family of squared distance functions over $M,\;\; d^2:\;M \times \r n   \to  \r{},$ is defined by $d^2(m,u) = d^2_u(m)= \Vert m - u \Vert^2.$
\end{definition}

The singularities of this family give a measure of the contacts of the
immersion with the family of hyperspheres of $\r n$ (\cite{Montesis}, \cite{P1}). Then, we observe that the function $d^2_u$ has a singularity in a point $m\in M$ iff
  $$
\displaystyle{\frac{\partial \phi}{\partial x}(x_0,y_0)
\cdot \big(\phi(x_0,y_0) -u \big)}=0, \quad \displaystyle{\frac{\partial \phi}{\partial y}(x_0,y_0) \cdot
\big(\phi(x_0,y_0) -u \big)} =0,
  $$
which is equivalent to say that the point $u$ lies in the normal subspace to $M$ at $m.$

\begin{definition}\label{2.1}
Given a surface $M$ immersed in $\r n,$ if the squared distance function $d^2_u$ has a degenerate singularity at $m$ then we say that the point $u \in \r n$ is a
 \textbf{focal center} at $m \in M.$ The subset of $\r n$ made of  all the focal centers for all the points of $M$ is called \textbf{focal set} of $M$ in $\r{}.$
 A hypersphere tangent  to $M$ at $m$  whose center lies in the focal set of $M$ at $m$ is said to be a \textbf{focal hypersphere} of $M$ at $m.$
\end{definition}

The focal set is classically known as the singular set of the normal exponential map $\exp_M:NM \rightarrow \r n$ (\cite{P1}, \cite{MGM-AM-SM-ES}). It is easy to see that the directions
of higher contacts of $M$ with the focal hyperspheres are those contained in the kernel of the quadratic form
  $$
\frac12\operatorname{Hess}(d^2_u)=g_m-(m-u)\cdot \alpha_m,
  $$
where $g_m$ and $\alpha_m$ are the first and second fundamental forms at $m,$ respectively.

In the remainder of this subsection, we will assume that $n=4.$
It follows from a general result of Montaldi \cite{montaldi} (and
also Looijenga's Theorem in \cite{Lo}) that for a residual set of
immersions $\phi: M\to \mathbb R^4$, the family $d^2$ is a generic
family of mappings. (The notion of a generic family is defined in terms of
transversality to submanifolds of multi-jet spaces, see for example
\cite{GG}.) We call these immersions, \textit{generic immersions}.

Among all the focal hyperspheres which lie in the singular subset of the focal set of $M,$
we have some special ones corresponding to distance-squared functions (from their centers) having (corank $1$)
singularities of type $A_k,$ with $k \geq 3.$ Here, we remind that an $A_k$ singularity is a germ of function
$ I\!\! R^2 \rightarrow I\!\! R$ which can be transformed by a local change of coordinates in $I\!\! R^2$ to
the germ of $x_1^2\pm x_ 2^{k+1},$ \cite {A1}.

\begin{definition}\label{2.2}
The centers of the focal hyperspheres of $M$ which have contact of type $A_k,$  $k \geq 3$ are called \textbf{($\mathbf{k}$-order) ribs}
and they determine normal directions called \textbf{rib directions}. The corresponding points in $M$ are known as \textbf{($\mathbf{k}$-order)
ridges} and the corresponding directions are called \textbf{strong principal directions}.
\end{definition}

The $k$-order ridges with $k\geq 4$ (i.e. the $A_k$ singularities of squared distance functions with $k \geq 4$) are the singular points of the ridges set. For a
generic immersion, the ribs form a stratified subset of codimension one in the focal set and the $k$-order ridges, $k \geq 4,$ form curves
with the $5$-order ridges as isolated points, \cite{Montesis}.
Other special kind of focal hyperspheres is made by those corresponding to squared distance functions that have corank $2$ singularities. In this case, all the
coefficients of the quadratic form Hess$(d^2_u)$ vanish.

\begin{definition}[\cite{S-C-F}]\label{2.3}
A focal center of $M$ at a point $m$ is said to be an \textbf{umbilical focus} provided the corresponding squared distance function has a
singularity of corank  $2$ at $m.$ A tangent $3$-sphere centered at an umbilical focus is called \textbf{umbilical focal hypersphere}.
\end{definition}

Montaldi proved in \cite{Montesis} the following relation between the (non radial) semiumbilic points and umbilical focal hyperspheres: {\it A point $m \in M$ is a
(non radial) semiumbilic if and only if it is a singularity of corank $2$ of some distance squared function on $M,$ in other words, it is a contact point of $M$
with some umbilical focal hypersphere at $m.$}

The corank $2$ singularities of distance-squared functions on generically immersed surfaces in $I\!\! R^{4}$ belong to the series
$D_k^{\pm}$ (see \cite{A1}). Moreover, on a generic surface, there are only $D_4^{\pm}$ singularities along regular curves with isolated $D_5.$

\section{The Taylor expansion of the exponential map}

 Let $M$ be an immersed submanifold  in $\r n$  and  $m\in M.$ We know that there is an open neighborhood $U_m$ of $0\in T_mM$ such that the exponential map
$\exp_m: U_m\to\r n$ is an one-to-one immersion. We recall also that $\exp_m(x)= \gamma_x(1),$ where $\gamma_x: [0,1]\to\r n$ is the geodesic in $M$ with initial
condition $\gamma_x(0)=m,\; \gamma_x'(0)=x\in U_m.$ We shall consider the Taylor expansion of $\exp_m$ around the origin of $T_mM.$ It will be written as
  $$
\exp_m(x)= m + I_m(x) + \frac12 Q_m(x)+
\frac1{6}K_m(x)+\dots,
  $$
where $I_m, Q_m, K_m$ are respectively linear, quadratic and cubic forms in $T_mM$ with values in $\r n.$

Our purpose is to write these forms in terms more familiar with the usual techniques of differential geometry. Let $x\in U_m$ and put $x= t v,$ where $t\in\r{}$ and
$v\in S^1(T_mM)$ is a unit vector. Then, as it is well known, $\exp_m(x)= \exp_m(tv)=\gamma_v(t).$ Therefore
  $$
\gamma_v(t)=m+ I_m(v)t+ \frac12Q_m(v)t^2 + \frac1{6}K_m(v)t^3+O(t^4).
  $$
Hence, $\gamma_v'(0)= v = I_m(v),$ so that $I_m: T_mM\to\r n$ is the inclusion. We also have $\gamma''_v(0)= Q_m(v)$ and $\gamma'''_v(0)= K_m(v).$

Now, $\gamma_v$ is a geodesic in $M$ and this implies that at every $t$ we have $\gamma''_v(t)\in N_{\gamma_v(t)}M.$ In fact, we have then
$\gamma_v''(t)=\alpha_{\gamma_v(t)}(\gamma'(t),\gamma'(t)).$ Hence,
  $$
Q_m(v)= \gamma''_v(0)=\alpha_m(v,v).
  $$

Thus, it is clear that the second order geometry of $M$ around $m$ is determined by the value at $m$ of the second fundamental form of $M.$ Let us study the third
order geometry.

Let $\xi\in T_mM.$ We may make the parallel transport of $\xi$ along the geodesic $\gamma_v$ in order to have a parallel vector field $X(t)$ along that geodesic.
This means that $X(0)=\xi,\; X(t)\in T_{\gamma_v(t)}M$ and $X'(t)\in N_{\gamma_v(t)}M.$ Then, we will have $X\cdot\gamma_v''=0.$ Differentiating, we get
  \begin{align*}
X\cdot\gamma_v'''&= -X'\cdot\gamma_v''= -X'\cdot\alpha(\gamma_v',\gamma_v')=-(D_{\gamma_v'}X)\cdot \alpha(\gamma_v',\gamma_v')\\
                 &=-\alpha(X,\gamma_v')\cdot \alpha(\gamma_v',\gamma_v').
  \end{align*}
Hence, by evaluation at $t=0$ we have
  $$
\xi\cdot K_m(v)=\xi\cdot\gamma'''_v(0)=-\alpha_m(\xi,v)\cdot
\alpha_m(v,v).
  $$
We observe thus that the tangential part of the third order geometry at $m$ depends only on the second order geometry at $m.$ Now, let $\zeta\in N_mM.$ As before, we
 define the vector field $Z(t)$ along the curve $\gamma_v$ as the parallel transport of $\zeta.$ Thus, for any $t$ we will have
 $Z(t)\in N_{\gamma_v(t)}M$ and $Z'(t)\in T_{\gamma_v(t)}M.$ Hence $Z'\cdot \gamma_v''=0.$ Thus
  $$
Z\cdot\gamma_v'''=(Z\cdot\gamma_v'')'=(Z\cdot\alpha(\gamma_v',\gamma_v'))'
=Z\cdot\big(\nabla_{\gamma_v'}\alpha\big)(\gamma_v',\gamma_v'),
  $$
because $Z$ and $\gamma_v'$ are parallel along $\gamma_v$ and  $(Z\cdot\alpha(\gamma_v',\gamma_v'))'=D_{\gamma_v'}(Z\cdot\alpha(\gamma_v',\gamma_v')).$

We have thus that $\zeta\cdot K_m(v)=\zeta\cdot\big(\nabla_v\alpha\big)(v,v).$ Having in mind that $(\nabla_v\alpha)(v,v)\in N_mM,$ we conclude that, for any
$u\in\r n$ and for any $x\in U_m,$ we have
  \begin{align}
u\cdot \exp_m(x)=& u\cdot m + u\cdot x + \frac12u\cdot\alpha_m(x,x)\\
                 &-\frac16 \alpha_m(u^\top,x)\cdot\alpha_m(x,x)+\frac16 u\cdot(\nabla_x\alpha)(x,x)+O(|x|^4). \nonumber
  \end{align}
Let us put $\alpha_m^\sharp= \sum_i t_i\otimes\alpha_m(t_i,\cdot),$ where $(t_1,\dots,t_k)$ is an orthonormal basis of $T_mM,$ and take the convention that if
$z\in \r n$ and $X\in T_mM$ then
   $$
z\cdot\alpha_m^\sharp(X)= \sum_i(z\cdot t_i)\alpha_m(t_i,X)=\alpha_m(z^\top,X),\quad
\alpha_m^\sharp(X)\cdot z= \sum_i t_i\big(\alpha_m(t_i,X)\cdot z\big).
   $$
Then
   \begin{align*}
\exp_m&(x)= m + x + \frac12\alpha_m(x,x)- \frac16\alpha^\sharp_m(x)\cdot\alpha_m(x,x)+\frac16 (\nabla_x\alpha)(x,x)+O(|x|^4),\\
\gamma_v(t)&= m + vt + \frac12\alpha_m(v,v)t^2+ \frac16\left( (\nabla_v\alpha)(v,v)-\alpha^\sharp_m(v)\cdot\alpha_m(v,v)\right)t^3+O(t^4).
   \end{align*}

This gives the \emph{geodesic deviation} $\Delta_v(t)$ defined by $v$ as
   $$
\Delta_v(t) =
 \gamma_v(t)- (m+vt)=  \frac12 \alpha_m(v,v)t^2 + \frac16((\nabla_v\alpha)(v,v)-\alpha^\sharp(v)\cdot\alpha_m(v,v))t^3 +
O(t^4).
  $$

Using the same technique, it is easy to compute higher order terms of these Taylor expansions, but we shall not use them here.

The tangential and normal components of the geodesic deviation are given by
\begin{align*}
u^\top\cdot \Delta_v(t) &= -\frac16 \alpha_m(u^\top,v)\cdot\alpha_m(v,v)t^3+ O(t^4),\\
u^\bot\cdot\Delta_v(t) &=  \frac12u^\bot\cdot\alpha_m(v,v)t^2
                 +\frac16 u^\bot\cdot(\nabla_v\alpha)(v,v)t^3+O(t^4). \nonumber
  \end{align*}

We see that the term of second order of the normal deviation is $\frac12\alpha_m(v,v)t^2,$ and this gives a geometric interpretation of the second fundamental form.
We will call its coefficient in $t^2$ the \emph{frontal deviation} of $M$ in  the direction $v\in T_mM.$ In the following we will give geometric interpretations to
the terms of third order.

\section{Applications to surface geometry}

In this section, $M$ will be a regular surface immersed in $\r n.$ Since the study is local we may assume that $M$ is orientable, so that there is a well defined
rotation of 90 degrees in $T_mM$ for each $m\in M.$ It will be given by the tensor field $J.$ We will focus here in the principal term of the tangential part of the geodesic deviation which is
   $$
-\frac16\alpha^\sharp(v)\cdot\alpha(v,v) t^3.
   $$
We decompose it into two components, one in the direction of $v$ and the other one in the direction of $Jv.$

\begin{definition}
We define the \textbf{frontal (geodesic) deviation} of $M$ in the (unit) direction $v$ by
   $$
-\frac16\alpha(v)\cdot\alpha(v,v).
   $$
The other component of this deviation, called \textbf{lateral (geodesic) deviation} of $M$ in the (unit) direction $v,$ is
given by
   $$
-\frac16\alpha(Jv,v)\cdot\alpha(v,v).
   $$
\end{definition}

\subsection{Lateral geodesic deviation of a surface in one direction}

Now we are going to give an additional interpretation to the lateral deviation. Suppose that $\gamma_v''(0)\neq 0.$  We consider the curve $\ovgam(t)$ obtained by
the orthogonal projection of $\gamma_v(t)$ over the affine subspace by $m$ generated by the orthonormal vectors
$e_1=\gamma_v'(0),\; e_2=\dfrac{\gamma_v''(0)}{\Vert\gamma_v''(0)\Vert}$ and $e_3=J\gamma_v'(0).$ That projection will be given, in the affine frame
$(m;\,e_1,e_2,e_3),$  by:
  $$
\ovgam(t)=((\gamma_v(t)-m)\cdot
e_1)e_1+((\gamma_v(t)-m)\cdot e_2)e_2+((\gamma_v(t)-m)\cdot e_3)e_3.
  $$
Thus, $\ovgam'(0)=v= e_1,$ and $\ovgam''(0)= \gamma_v''(0)=\Vert\alpha(v,v)\Vert e_2,$ and
   \begin{align*}
\ovgam'''(0)\cdot e_1&= \gamma'''_v(0)\cdot e_1= -\Vert\alpha_m(v,v)\Vert^2,\\
\ovgam'''(0)\cdot e_2&= \frac{(\nabla_v\alpha)(v,v)\cdot\alpha(v,v)}{\Vert\alpha(v,v)\Vert},\\
\ovgam'''(0)\cdot e_3&=-\alpha(Jv,v)\cdot\alpha(v,v).
   \end{align*}

Therefore $\ovgam'(0)\times \ovgam''(0)= \Vert\alpha(v,v)\Vert e_3,$ hence the torsion of $\ovgam$ at $t=0$ is given by:
  \begin{align*}
\overline{\tau}=&\dfrac{(\overline{\gamma}'(0)\times\overline{\gamma}''(0))\cdot\overline{\gamma}'''(0)}{|\overline{\gamma}'(0) \times\overline{\gamma}''(0)|^2}\\
               =&-\dfrac{\alpha(Jv,v)\cdot\alpha(v,v)}{\Vert\alpha(v,v)\Vert}.
  \end{align*}

Now, the curvature of $\overline\gamma(t)$ is given by $\overline\kappa(0)=\Vert\overline\gamma''(0)\Vert=\Vert\alpha(v,v)\Vert.$ Then, the lateral deviation of $M$
in the direction of the unit vector $v\in T_mM$ is
  $$
-\frac16\alpha(Jv,v)\cdot\alpha(v,v)=\frac16{\overline\kappa}(0)\,\overline{\tau}(0),
  $$

Finally, we know that if $\kappa_v(t)$ denotes the curvature of the geodesic
$\gamma_v(t)$ then $\kappa_v(t)^2=\Vert\alpha_{\gamma_v(t)}(\gamma_v'(t),\gamma_v'(t))\Vert^2.$ Therefore,
   \begin{align*}
\kappa_v(t)&\kappa_v'(t)=\alpha_{\gamma_v(t)}(\gamma_v'(t),\gamma_v'(t))\cdot \nabla_{\gamma_v'}\big(\alpha_{\gamma_v}(\gamma_v',\gamma_v')\big)(t)\\
=&\alpha_{\gamma_v(t)}(\gamma_v'(t),\gamma_v'(t))\cdot(\nabla_{\gamma_v'}\alpha)(\gamma_v',\gamma_v')(t).
   \end{align*}
Evaluating at $t=0$ we get
   $$
\kappa_v(0) \kappa'_v(0)= \alpha(v,v)\cdot(\nabla_v\alpha)(v,v),
   $$
so that if we denote $\kappa_v= \kappa_v(0)$ and
$\kappa'_v=\kappa'_v(0),$ we have
   $$
\kappa_v\kappa'_v=\alpha(v,v)\cdot(\nabla_v\alpha)(v,v).
   $$
and it measures the geodesic ratio of change of the normal curvature
in the  direction $v.$

\subsection{Retard of a geodesic with respect to the tangent vector}

In this section we give an interpretation to the frontal deviation.

Let $t\mapsto m + tv$ be the geodesic in $\r n$ with same initial condition as $\gamma.$ We can approximate $\gamma_v$ to order two by a curve $\beta$ that describes,
 with velocity $v,$ a circle of radio $R=\frac1{\gamma_v''(0)}=\frac1{\Vert\alpha(v,v)\Vert}$ that lies on the affine plane by $m$ generated by $\gamma'(0)$ and
 $e_2=\frac{\gamma_v''(0)}{\Vert\alpha(v,v)\Vert}.$ The equation of this curve is
   $$
\beta(t)= m + R\sin \frac tR\; e_1+(R-R\cos \frac tR)e_2.
   $$
The retard of the projection of $\beta(t)$ on the tangent plane with respect to the curve $m+ tv$ is given by:
   \begin{align*}
(R\sin\frac tR - t) v+\dots&= -\frac16\frac{t^3}{R^2}+\dots\\
                     &=-\frac16 \alpha(v,v)\cdot\alpha(v,v)t^3 + \dots
   \end{align*}
This explains why the frontal deviation depends only on the second order geometry: it is a consequence of the curvature of $\gamma_v$ together with the fact that it
 is parameterized by arc-length.

\subsection{Extremal directions of the frontal geodesic deviation}

The frontal geodesic deviation of $M$ in the direction $v$ depends essentially in  the norm of the second fundamental form. Hence, the extremal directions of this
deviation are the directions where its derivative vanishes. We are going to find these directions when $M$ is a surface. To simplify calculations, we differentiate
the squared norm instead of the norm itself.

We know that
  $$
\eta(\theta)= \alpha(v,v)= H+B\cos2\theta +C\sin 2\theta,
  $$
where $v= \cos\theta \,t_1+ \sin \theta \, t_2.$ The derivative of the squared norm of $\eta(\theta)$ vanishes iff:
  $$
(H+B\cos2\theta +C\sin 2\theta)\cdot(-B\sin2\theta+ C\cos2\theta)=0.
  $$
And this is equivalent to:
  $$
-hb\sin2\theta+hc\cos2\theta+(cc-bb)\sin2\theta\cos2\theta+bc\cos^22\theta-bc\sin^22\theta=0,
  $$
where we have put $hb= H\cdot B,\; bb= B\cdot B,$ etc. Now, putting $p=\tan\theta,$ the extremal directions of the frontal deviation are given by the solutions of
the following equation:
\begin{eqnarray}\nonumber
p^4(-hc+bc)+p^3(-2cc+2bb-2hb)+ p^2(-2bc-4bc)&&\\\nonumber
+ p(-2hb+2cc-2bb)+bc+hc&=&0.
\end{eqnarray}

This equation could serve for computing numerically those directions and the corresponding lines of extremal frontal deviation.
\begin{example}
  Let $M$ be a surface immersed in $\r 5$ and $\stackrel{\rightarrow}{x}:U\subset\r 2 \to M$ be a chart defined in $M,$ $U$ be an open set, where:
\[
\begin{array}{rcll}
\stackrel{\rightarrow}{x}:&U &\longrightarrow& M \\
                      &(u,v) &\longrightarrow &\left(u^2v^2, u+v, u-v, \frac{u^2+v^2}{2}, \frac{u^2-v^2}{2}\right).

\end{array}
\]
The  figure $1$ has been made with the program \cite{AMA2}. The program draws the lines that are at each point tangent to one of the two or four directions of
extremal frontal geodesic deviation. The thick line is the discriminant curve separating the regions where there are two such directions at each point, from those
where there are four.
\begin{figure}[ht]
\begin{center}
  \includegraphics[width=3cm]{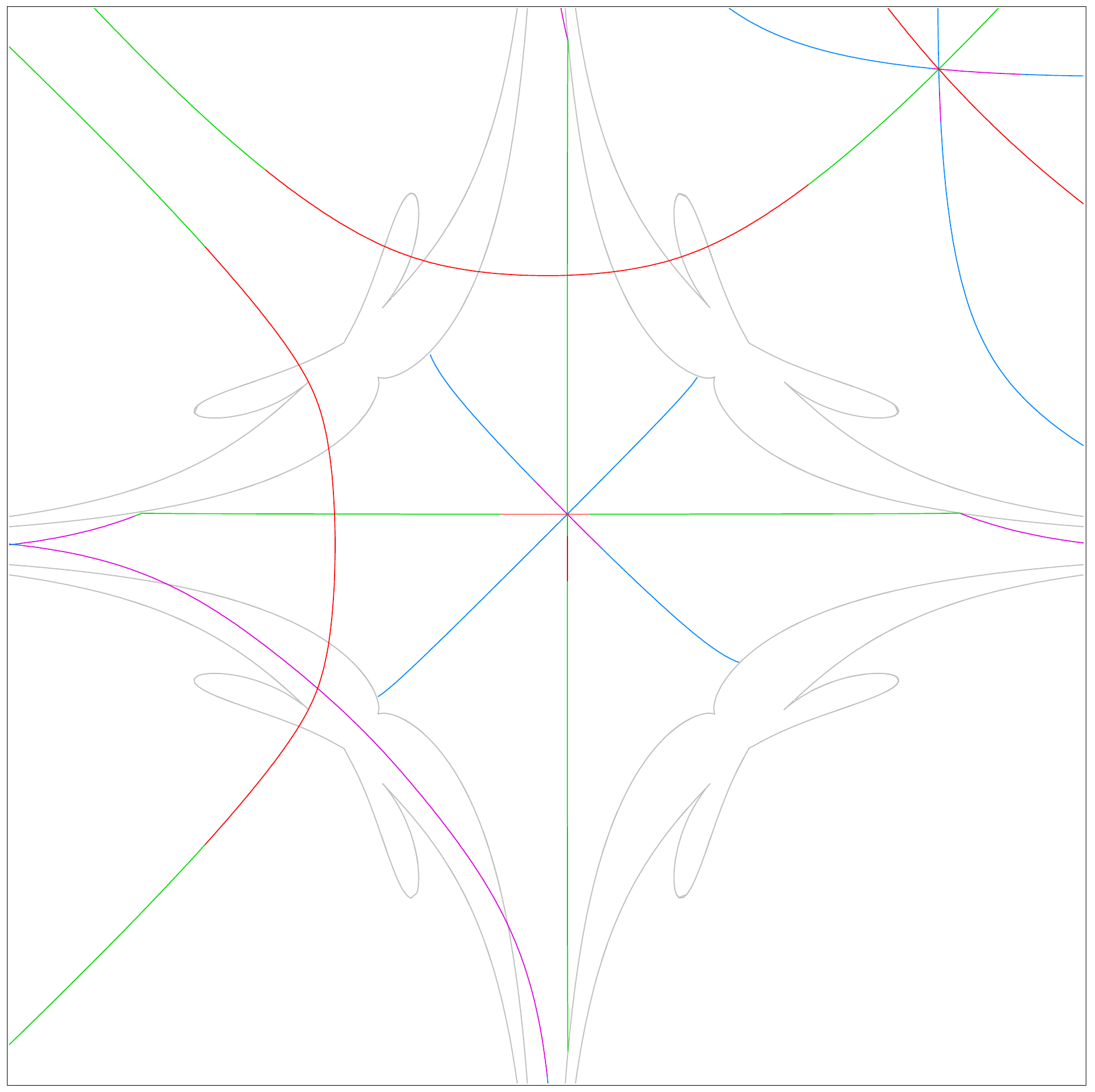}\\
  \caption{$\stackrel{\rightarrow}{x}(u,v)=\left(u^2v^2, u+v, u-v, \frac{u^2+v^2}{2}, \frac{u^2-v^2}{2}\right)$}
\end{center}
\end{figure}
\end{example}

For giving a feeling about the difficulty, if not impossibility, of making effective computations via Monge charts not known
beforehand, we recall the following. Suppose that the initial data of the problem is given in the most usual manner, that is by a
chart of the surface $S$ as $ X(u,v)= (g_1(u,v), \dots, g_5(u,v)).$ The task for obtaining, for instance, the asymptotic directions at
a \underline{single} point  $ p= X(u_0, v_0) $ begins by computing a Monge chart around $p.$ That is, we first compute the basis
$(A=X_u(u_0,v_0), B=X_v(u_0,v_0))$ of  $T_pS$ and a basis $(u_1,u_2, u_3)$ of the normal space to $S$ at $p.$ Let $(u,v)$ be a
point near $(u_0,v_0).$ The Monge parameterization defined around $(u_0,v_0)$ by the basis $(A,B,u_1,u_2,u_3)$   is such that, given
the pair $(x,y)$ near zero, there must be some pair $(u,v)$ and numbers $f_1, f_2, f_3$ near zero satisfying
    $$
X(u,v) - p =  x A + y B + f_1 u_1 + f_2 u_2 + f_3 u_3 .
    $$

The five components of this equality will give five equations for the five unknowns $u, v, f_1, f_2, f_3.$  So, we obtain functions
$f_i(x,y),\;i=1,\dots,3$ such that the map $(x,y)\to x A + y B + f_1(x,y) u_1 + f_2(x,y) u_2 + f_3(x,y) u_3$ is the desired Monge chart.

We note that while those equations are linear in the unknowns $f_1, f_2$ and $f_3,$ they may have any form in $X(u,v).$
Therefore, unless $X$ have a very simple expression, the task is hopeless. As an appendix, we offer as proof a notebook showing
that a task so simple as computing a Monge chart for a sphere in $R^3$ from the usual parameterization is too much for
Mathematica$^\circledR.$ The same occurs even for surfaces in $R^3$ given through polynomials in $u$ and $v.$  For instance, the minimal Bour surface given by
    $$
X(u,v) = \big((u^2(2- u^2 + 6v^2) - v^2(2+v^2))/4,\; uv(v^2 - u^2-1),\; 2u(u^2 - 3v^2)/3).
    $$
Of course, this is not intended as a  criticism on that manificent package for symbolic and numeric computation.

Note also that Monge charts are used mainly in the form of power series for the functions $f_1, f_2,...$, so that the trick for
obtaining geometric results only works for the point $(x,y)= (0,0).$ For instance, if we would compute the integral curves of
asymptotic directions as in (\cite{M}, \cite{MRR}), we would need to compute a Monge
chart for each point where there would be necessary to compute the asymptotic directions according with the chosen ordinary
differential equation algorithm, usually several times for each point of the curve effectively computed.

\subsection{Relation between extremal frontal and lateral geodesic directions}

Now we will find the directions $\theta$ of $T_mM$ where the values of the lateral tangent deviation coefficient $-\alpha(Jv,v)\cdot \alpha(v,v)$ are extremal. These
 directions are the directions where the derivative of $-\alpha(Jv,v)\cdot \alpha(v,v)$ vanishes. In terms of $H,B$ and $C$ we have:
    \begin{align*}
\alpha(Jv&,v)=  \alpha(t_2\cos\theta-t_1\sin\theta,t_1\cos\theta+t_2\sin\theta  ) \\
     & = -\dfrac{\sin2\theta}{2}\,b_1+ \cos2\theta \,b_3 + \dfrac{\sin2\theta}{2}\,b_2.
   \end{align*}

Since $b_1= H+B$ and $b_2=H-B,$ we have
  $$
\alpha(Jv,v)=-B\sin2\theta +C\cos2\theta = \frac12 \eta(\theta)'.
  $$
Finally:
  $$
 \begin{array}{rcl}
   \alpha(Jv,v)\cdot\alpha(v,v) & = & (-B\sin2\theta +C\cos2\theta)\cdot(H + B\cos2\theta +C\sin2\theta) \\[2mm]
     & = &-hb\sin2\theta + (cc-bb)\sin2\theta\cos2\theta-bc\sin^22\theta\\[2mm]
     && +cb\cos^22\theta + ch\cos2\theta.
 \end{array}
  $$

Note that $\alpha(Jv,v)\cdot\alpha(v,v)=\frac12\eta(\theta)\cdot \eta(\theta)'= \frac14 \big(\eta(\theta)\cdot\eta(\theta)\big)'.$ In other words, the lateral
deviation is proportional to the derivative of the squared norm of the frontal deviation. With this, we have proved the following proposition.

\begin{proposition}
The directions where the lateral deviation vanishes are the extremal directions of the frontal geodesic deviation.
\end{proposition}

%
%

Notice that by using the  expression $0=\frac12\eta(\theta)\cdot \eta(\theta)'$ we characterize the extremal frontal geodesic directions as the tangent directions where the distance of the ellipse to the origin is extremal. In other words, a tangent direction is an extremal frontal direction when the corresponding point of the ellipse belongs to a hypersphere of $N_mM$ centered at the origin and tangent to the ellipse at this point. This guarantees the existence of at least 2 extremal directions.

\begin{remark}
Let $M$ be a $k$-dimensional submanifold in $\r n.$ For a given point $m\in M$ and a given unit vector $v\in T_mM$ there exists a unique geodesic $\gamma:I\to M$
with $\gamma(0)=m$ and $\gamma'(0)=v$ and a unique normal section $\beta:I\to M$ associated to $m$ and $v.$ Then $\gamma'(0)=\beta'(0)=v$ and we know that
$\gamma''(0)=\beta''(0)=\alpha(v,v).$ On the other hand, we say that two regular curves ${\gamma},\,{\beta}$ with a point in common
${\gamma}(t_0)= {\beta}(t_0)=m,$ have a {\bf contact of order $\mathbf{k}$} in $m$ iff there exists a parametrization of the curves where the first $k-1$
derivatives coincides in that point, that is:
\begin{center}
\begin{tabular}{rcl}
$\gamma^{(i)}(t_0)$ & = & $\beta^{(i)}(t_0),\quad i=1,\dots, k-1,$ \\
$\gamma^{(k)}(t_0) $& $\neq$ & $\beta^{(k)}(t_0).$
\end{tabular}
\end{center}

Then we observe that the contact between the geodesic $\gamma$ and the normal section $\beta$ is at least of order $2.$ In (\cite{Chen}) it is proved that the
contact between $\gamma$ and $\beta$ is at least of order 3, that is, $\gamma^{'''}(0)=\beta^{'''}(0)$ if and only if $\alpha(v,Jv) \cdot \alpha(v,v)=0.$ Then the
contact between $\gamma$ and $\beta$ is at least of order 3 on a surface $M$ in $\r m$ in and only if $v$ is an extremal frontal separation direction.
\end{remark}

\subsection{Extremal directions in $\r 3$}
Now suppose that $M$ be a surface immersed in $\r 3.$ In this case the curvature ellipse is reduced to a segment. Then there exist $ a,b,c\in \r{}$ such that
$H=hN,$ $B=bN$ and $C=cN,$ where $N$ is the unit normal of $M.$ Hence:
  $$
\alpha(t,t) =  (h+b\cos2\theta+ c\sin2\theta) N,
  $$
where $t=t_1\cos\theta+t_2\sin\theta.$

In the case of frontal separation, the extremal directions are those that make the derivative of the squared norm of the second fundamental form to vanish. Then, in
 this case:
  $$
4(h+b\cos2\theta +c\sin 2\theta)(-b\sin2\theta+ c\cos2\theta)=0.
  $$

We have two possibilities:
\begin{itemize}
  \item[1)]$h+b\cos2\theta +c\sin 2\theta=0,$ then $\theta$ is a asymptotic direction.
  \item[2)]$-b\sin2\theta+ c\cos2\theta=0,$ then $\theta$ is a principal direction.
\end{itemize}
In the case of the lateral deviation, the  extremal directions are given by the equation:
  $$
-hb\cos2\theta+ (cc-bb)(\cos^22\theta -\sin^22\theta)-4bc\cos2\theta\sin2\theta-hc\sin2\theta=0.
  $$

Another expression of this can be obtained as follows. Let $k_1,k_2$ be the principal curvatures. The Euler formula says that the normal curvature of $M$ at $m$ in
the direction determined by $\theta$ is given by  $k_n(\theta)=k_1\cos^2\theta + k_2\sin^2\theta.$ Then $k'_n(\theta)=(k_2-k_1)\sin2\theta.$

In this case, we study the directions where the derivative of $k_n\, k'_n=0$ vanishes. We know that
$k_n\, k_n'  =  (k_1\cos^2\theta + k_2\sin^2\theta)(k_2-k_1)\sin2\theta.$ Differentiating this equation we have:
  $$
(k_2-k_1)(2k_1\cos2\theta\cos^2\theta+2k_2\cos2\theta\sin^2\theta+ (k_2-k_1)\sin^22\theta)=0.
  $$
Now, putting $p=\tan \theta,$ this is
  $$
(k_2-k_1)(-k_2p^4+ 3p^2(k_2-k_1)+ k_1)= 0.
  $$
Solving this equation at a non-umbilic point, the normal curvatures of the extremal directions of the lateral deviation are given by:
  $$
k_n=k_1\cos^2\theta+ k_2\sin^2\theta= \dfrac{k_2k_1-3k^2_2\pm k_2\sqrt{9(k_2^2+ k_1^2)-14k_1k_2}}{3k_1-5k_2\pm\sqrt{9(k_2^2+ k_1^2)-14k_1k_2 }}.
  $$

One may verify from these values that the extremal lateral deviation directions are different from all of the special directions on surfaces that we know of, namely
asymptotic, principal, arithmetic and geometric mean \cite{GS} or characteristic (harmonic mean) \cite{T}.

\begin{remark}
Suppose that $M$ is a minimal surface immersed in $\r 4.$ In this case we know that $H=0,$ then the extremal frontal deviation lines coincides with the lines of
axial curvature defined in \cite{RG-JS}.
\end{remark}

\subsection{Normal curvature and torsion}
In this section, we will show how the Taylor expansion of the exponential map allows us to obtain easily an intrinsic expression for the normal torsion of a surface
in $\r 4$ in a tangent direction.

The definition of normal torsion at a point along one direction was given by W. Fessler in \cite{W}. Let $M$ be a surface immersed in $\r 4,$ $m\in M,$ and
$0\ne v\in T_mM.$ Consider the affine subspace of $\r 4$ which passes by $m$ and is generated by $v$ and $N_mM.$ The intersection of this subspace with $M$ is a
curve that passes by $m,$ called the \it normal section \rm of $M$ determined by $v.$ The curvature and torsion of this curve, as a curve in that Euclidean affine
3-space, is the \it normal curvature \rm and \it normal torsion \rm of the surface $M$ in the direction $v,$ respectively.

The inverse image by $\exp_m$ of the normal section of $M$ in the direction given by the unit vector $v\in T_mM$ is a curve in $T_mM$ whose Taylor expansion may be
written as  $\beta(t)=v\, t+ \frac{1}{2}a\, Jvt^2+ \frac{1}{6}c\,t^3+\ldots,$ where $a\in\r{}\;,c\in T_mM,$ and $v\cdot v=1.$

We have:
  \begin{align*}
(\exp(\beta(t))-m)\cdot Jv&=Jv\cdot \beta(t)-\frac{1}{6}\alpha(Jv,\beta(t))\cdot\alpha(\beta(t),\beta(t))+\ldots\\
                          &=\frac{1}{2}at^2+\frac{1}{6}(Jv\cdot c)t^3 -\frac{1}{6}\alpha_m(Jv,v)\cdot \alpha_m(v,v)t^3+O(t^4).
  \end{align*}
Now since $\exp(\beta(t))$ is a normal section, $\exp(\beta(t))-m$ will belong to the subspace $<v,N_mM>.$ Hence $(\exp(\beta(t))-m)\cdot Jv=0$ and this implies that
 $a=0$ and $Jv\cdot c= \alpha_m(Jv,v)\cdot \alpha_m(v,v).$ Therefore
   $$
\beta(t)= vt+\frac16\big((c\cdot v) v + \alpha_m(Jv,v)\cdot\alpha_m(v,v)Jv\big)t^3+O(t^4).
   $$
We put $\mu(t)$ to denote the terms up to the third order in $t$  of $\exp(\beta(t))-m.$ We compute the component of $\mu(t)$ along $v$
   \begin{align*}
v\cdot \mu(t)=&v\cdot\beta(t)-\frac16\alpha_m(v,\beta(t))\cdot\alpha_m(\beta(t),
\beta(t))\\
&= t+\frac16\big(v\cdot c - \alpha_m(v,v)\cdot\alpha_m(v,
v)\big) t^3
   \end{align*}

As for the normal component of $\mu(t)$, it is given by
   $$
  \mu(t)^\bot=\frac12 \alpha_m(v,v)t^2+\frac16(\nabla_{v}\alpha)(v,v) t^3.
   $$

In the following, the formulas for $\mu$ and its derivatives will have two components; the first one is the tangential component in the direction $v$ (the tangential
 component in the direction $Jv$ is zero); the second is the normal part which belongs to $N_mM.$
  \begin{align*}
\mu(t)&=\Big(t+\frac16\big(v\cdot c - \Vert\alpha_m(v,v)\Vert^2\big) t^3\;, \;
\frac12 \alpha_m(v,v)t^2+\frac16(\nabla_{v}\alpha)(v,v) t^3\Big),\\
\mu'(t)&=\Big(1+\frac12\big(v\cdot c -\Vert\alpha_m(v,v)\Vert^2\big)t^2\;, \;\alpha_m(v,v)t+\frac{1}{2}(\nabla_v\alpha)(v,v) t^2\Big),\\
\mu''(t)&=\big((v\cdot c -\Vert\alpha_m(v,v)\Vert^2)t\;,\;\;\alpha_m(v,v)+(\nabla_v\alpha)(v,v) t\big),\\
\mu'''(t)&=\big(v\cdot c -\Vert\alpha_m(v,v)\Vert^2\;,\;\;(\nabla_v\alpha)(v,v)\big).
  \end{align*}
We evaluate the last three formulas at $t=0,$ and get
   $$
\mu'(0)=(1,0),\;\; \mu''(0)= (0,\alpha_m(v,v)),\;\; \mu'''(0)=\big(v\cdot c -\Vert\alpha_m(v,v)\Vert^2,\;(\nabla_v\alpha)(v,v)\big).
   $$

Now it is easy to show that $\mu'(0)\times\mu''(0)=J\alpha(v,v)$ from which we have
   $$
(\mu'(0)\times\mu''(0))\cdot\mu'''(0)=J\alpha(v,v)\cdot (\nabla_v\alpha)(v,v).
   $$
Therefore the normal torsion of $M$ at $m$ in the direction $v\in T_mM$ is given by:
  \begin{align*}
\tau_v &= \frac{J\alpha(v,v)\cdot
(\nabla_v\alpha)(v,v)}{\alpha(v,v)\cdot
\alpha(v,v)}=\frac{J\gamma''_v\cdot\gamma'''_v}{\gamma''_v\cdot
\gamma''_v}(0).
  \end{align*}
where in the last formula $\gamma_v$ is the geodesic with initial
condition $v.$

The normal curvature in the same direction is $\kappa_v=\Vert \mu''(0)\Vert= \Vert\alpha_m(v,v)\Vert.$

\section{Applications to contact theory}

\subsection{Directions of high contact with $3$-spheres in $\r 4$}

Let $M$ be a surface immersed in $\r 4,\; m\in M$ and $0 \neq u \in\r 4.$  We will denote by $d_{3,u}$ the third order approximation of the function
$f: T_mM\to \r{}$ defined as $f(x) = h(x) - h(0),$ where $h(x) = d(exp_m(x), m+u)^2,$ that is

  \begin{align*}
d_{3,u}(x)=&-2u\cdot x+x\cdot x-u\cdot \alpha(x,x)+\frac13\alpha(u^\top,x)\cdot\alpha(x,x)\\
           &-\frac13 u^\bot\cdot(\nabla_x\alpha)(x,x),
  \end{align*}
where, for brevity, we have put $\alpha(x,x)$ instead of $\alpha_m(x,x).$

From definition \ref{2.2} it is known that $u$ determines a \it rib direction \rm  at $m$ if and only if the following conditions are true:
\begin{itemize}
\item [(i)] $u\in N_mM.$
\item [(ii)] There is some $x\in T_mM,\; x\ne 0,$ such that $g(x,\cdot)-u\cdot\alpha(x,\cdot)=0.$
\item[(iii)] $d_{3,u}(x)=0.$
\end{itemize}

This vector $x$ defines a \it strong principal direction \rm at $m$ i.e. a direction of at least $A_k$ contact, $k \geq 3,$ with the
corresponding focal hypersphere, \cite{P1}.

\begin{theorem}
If a vector $0\ne x\in T_mM$ defines a strong principal direction then it satisfies the following conditions:
   \begin{enumerate}
\item $\alpha(x,x) \ne 0.$
\item $J\alpha(x,x)\cdot\alpha(x,Jx)\ne 0\quad$ or $\quad \alpha(x,Jx)=0.$
\item $\det\big(\alpha(x,Jx),(\nabla_x\alpha)(x,x)\big)=0,$
   \end{enumerate}
where the determinant is meaningful because both vectors belong to $N_mM,$ whose dimension is two.
\end{theorem}

\begin{proof}
Assume that $0\ne x\in T_mM$ defines a strong principal direction. Then there exists a rib direction $u \in \r4$ satisfying properties
(i)-(iii). Condition (i) says that $u^\top= 0.$ Since $(x,Jx)$ is a basis of $T_mM$ condition (ii) is equivalent to the following two conditions
   $$
 x\cdot x= u\cdot\alpha(x,x),\qquad u\cdot\alpha(x,Jx)=0.
   $$
Since $x\ne 0,$ the first one requires that $\alpha(x,x)\ne 0.$ Therefore we can put
   $$
u= p\alpha(x,x) + q J\alpha(x,x)
   $$
for some $p,q\in \r{}.$ Then $u\cdot\alpha(x,x)= x\cdot x= p\Vert \alpha(x,x)\Vert^2,$ that is
   $$
p=\frac{x\cdot x}{\Vert \alpha(x,x)\Vert^2},
   $$
and
   $$
u\cdot \alpha(x,Jx)=0=\frac{(x\cdot x)\alpha(x,x)\cdot \alpha(x,Jx)}{{\Vert \alpha(x,x)\Vert^2}}+  q J\alpha(x,x)\cdot \alpha(x,Jx).
   $$
Hence, if $J\alpha(x,x)\cdot \alpha(x,Jx)\ne 0$ we can solve this for $q.$ Otherwise we must have
   $$
J\alpha(x,x)\cdot \alpha(x,Jx)= \alpha(x,x)\cdot \alpha(x,Jx)=0,
   $$
but since $\alpha(x,x)\ne 0$ and $J\alpha(x,x)\ne 0$ we conclude that $\alpha(x,Jx)=0.$ So, in any case condition \it 2 \rm is satisfied.

Also, if (i) and (ii) are satisfied, then $d_{3,u}(x)= -\frac13 u\cdot(\nabla_x\alpha)(x,x)$ and this must be zero.
Therefore, the non-zero vector $u\in N_mM$ must be orthogonal to
$\alpha(x,Jx)$ and $(\nabla_x\alpha)(x,x)\in N_mM.$ Since $\dim N_mM= 2,$ we conclude that these two vectors must be linearly dependent, i.e.
  $$
\det\big(\alpha(x,Jx),(\nabla_x\alpha)(x,x)\big)=0.
  $$
and this is condition \it 3. \rm
\end{proof}\hskip 6cm $\square$

Condition \it 3 \rm leads to an equation of degree $5$ which generically gives at most $5$ strong principal directions. That equation was first obtained by
M. Montaldi in \cite{M}, but note that he uses a Monge chart and his equations are opaque in the sense that they are not given in terms geometrically recognizable.
Conversely, we  have

\begin{theorem}
If a vector $x\in T_mM$ satisfies the following conditions:
   \begin{enumerate}
\item $\alpha(x,x) \ne 0,\quad J\alpha(x,x)\cdot\alpha(x,Jx)\ne 0,\quad
  \det\big(\alpha(x,Jx),(\nabla_x\alpha)(x,x)\big)=0,$

or

\item $\alpha(x,x) \ne 0,\quad \alpha(x,Jx)= 0,\quad \{\alpha(x,x)\cdot(\nabla_x\alpha)(x,x)=0$ or $J\alpha(x,x)\cdot (\nabla_x\alpha)(x,x)\ne 0\}.$

\end{enumerate}
then it defines a strong principal direction.
\end{theorem}

\begin{proof}
Suppose that $x$ satisfies \it 1. \rm Then, as we have seen, there is a non-vanishing vector $u$ that satisfies (i) and (ii). But then $u\cdot\alpha(x,Jx)=0,$ from
which we conclude that $u$ is orthogonal to  $(\nabla_x\alpha)(x,x),$ because by the second and third conditions this vector is a multiple of $\alpha(x,Jx)\ne 0$ and
 this leads to (iii).

Now, suppose that $x$ satisfies \it 2. \rm Then, for any value of $r\in\r{}$ we have that
   $$
u=\frac{x\cdot x}{\Vert \alpha(x,x)\Vert^2}\alpha(x,x)+ r\,J\alpha(x,x)
   $$
satisfies (i) and (ii). The condition (iii) is now
   $$
\left(\frac{x\cdot x}{\Vert \alpha(x,x)\Vert^2}\alpha(x,x)+ r\,J\alpha(x,x)\right)\cdot (\nabla_x\alpha)(x,x)=0
   $$
If $\alpha(x,x)\cdot(\nabla_x\alpha)(x,x)=0$ then the choice $r=0$ solves the existence of the needed vector $u.$ If
$J\alpha(x,x)\cdot (\nabla_x\alpha)(x,x)\ne 0,$ then we can solve the equation for $r$ and find again the vector $u.$
\end{proof} \hskip 0.5cm $\square$

The program \cite{AMA1} can show the strong principal directions and curves.

Now we are going to show the manner in which the more complicate condition, namely condition \it 3 \rm of Proposition 4.1  may be computed. First of all, it is clear
 that
   $$
\det\big(\alpha(x,Jx),(\nabla_x\alpha)(x,x)\big)= \det\big(t_1,t_2,\alpha(x,Jx),(\nabla_x\alpha)(x,x)\big),
   $$
where the last determinant assumes that the vectors are in $\r4.$
Now, assuming that in the following $x$ denotes an extension of $x$ in a neighborhood of $m$, we will have
   \begin{align*}
(\nabla_x\alpha)&(x,x)= \Big(D_x\big(\alpha(x,x)\big)\Big)^\bot- 2\alpha(\nabla_xx,x)\\
&=(D_x\alpha)(x,x)^\bot + 2\alpha(D_xx-\nabla_x x,x)= (D_x\alpha)(x,x)^\bot.
   \end{align*}
Thus the condition becomes
   $$
\det\big(t_1,t_2,\alpha(x,Jx),(D_x\alpha)(x,x)\big)=0,
   $$
because the tangent component of $(D_x\alpha)(x,x)$ is canceled by the presence of the tangent basis $(t_1,t_2)$ in the determinant.
Now, if we put $x= \cos\theta t_1+ \sin\theta t_2$ and denote
   $$
q = t_2\cdot D_{t_1}t_1,\qquad r= t_2\cdot D_{t_2}t_1.
   $$
we will have
   \begin{align*}
(D_x\alpha)&(t_1,t_1)= D_xb_1- 2\alpha(D_xt_1,t_1)=D_xb_1-2(t_2\cdot D_xt_1)b_3\\
=&\cos\theta(D_{t_1}b_1-2qC)+\sin\theta(D_{t_2}b_1-2rC),
   \end{align*}
because $t_1\cdot D_xt_1=0.$ In the same manner we obtain
   \begin{align*}
&(D_x\alpha)(t_1,t_2)= \cos\theta(D_{t_1}b_3+2qB)+\sin\theta(D_{t_2}b_3+2rB)\\
&(D_x\alpha)(t_2,t_2)= \cos\theta(D_{t_1}b_2+2qC)+\sin\theta(D_{t_2}b_2+2rC),
   \end{align*}

Then,
   \begin{align*}
(D_x\alpha)&(x,x)= \cos^3\theta(D_{t_1}b_1-2qC)\\
&+ \sin\theta\cos^2\theta\big(D_{t_2}b_1+ 2 D_{t_1}b_3+ 4qB- 2rC\big)\\
&+ \sin^2\theta\cos\theta\big(D_{t_1}b_2+ 2 D_{t_2}b_3+4rB+ 2qC\big)\\
&+ \sin^3\theta(D_{t_2}b_2+2rC).
   \end{align*}
Since
  $$
  \alpha(x,Jx)=-B\sin 2\theta + C\cos2\theta,
  $$
the determinant may be written as a homogeneous polynomial of fifth degree in the variables $\cos\theta$ and $\sin\theta.$ If we put $p=\tan\theta$ it gives in
general an equation of fifth degree in $p$ that results in at most five strong principal directions (or an infinity if all its coefficients vanish). Then, by using
the last Proposition one can get the respective ribs.

Let $x\ne 0$ be a unit vector obtained by solving the fifth degree equation and put $b=b_{x}=\alpha(x,x)$ and $c=\alpha(x,Jx).$ Let us suppose in addition that
$b\ne 0$ and $Jb\cdot c\ne 0.$ Then, the conditions of Proposition 4.2,1 are satisfied and we will have that the corresponding rib direction is determined by
   $$
u= \frac{b}{\Vert b\Vert^2}-\frac{b\cdot c}{\Vert b\Vert^2 Jb\cdot c}Jb.
   $$

If $b\ne 0$ and  $c=0$ then
   $$
u= \frac{b}{\Vert b\Vert^2}.
   $$

In the first case, suppose that $\kappa'= \kappa'_{x}\ne 0.$ Then $c$ is a multiple of $n=(\nabla_{x}\alpha)(x,x),$ so that if $\kappa= \kappa_{x}$ we may write
   \begin{align*}
u=& \frac{b}{\kappa^2}- \frac{b\cdot n}{\kappa^2 Jb\cdot n}Jb\\
&= \frac{b}{\kappa^2}-\frac{\kappa'}{\kappa^3\tau}Jb,
   \end{align*}
where $\tau$ is the normal torsion of $M$ at $m$ in the direction
$x$.

From definition \ref{2.3} it is known that $u$ determines an \it umbilic direction \rm  at $m$ if and only if the following conditions are true:
\begin{itemize}
\item [(i)] $u\in N_mM.$
\item [(ii)] $g(x,y)-u\cdot\alpha(x,y)=0,$ for any $x, y \in T_mM.$
\end{itemize}

In this case we have a singularity of corank $2$ of the distance squared function on $M$ at $m,$ i.e. in this point the surface has at least $D_k$ contact,
$k \geq 4,$ with the corresponding umbilic focal hypersphere \cite{Montesis}. If $m$ is umbilic then there is some vector $b\in N_mM$ such that we have at $m$ that
$\alpha= b\otimes g.$ If $b=0,$ there are no umbilic directions at $m$. Otherwise, all vectors $u\in N_mM$ in the affine line given by the equation $u\cdot b= 1$
determine umbilic directions. The remaining cases are comprised in the following result, where we have reworded the theorem given in \cite{RF-F}.

\begin{theorem}
Let $m\in M$ be a non umbilic point. There is a vector $u\in N_mM$ determining an umbilic direction at $m\in M$ if and only if $m$ is a semiumbilic non-inflection
point.
\end{theorem}

\begin{proof}
Assume that $u\in N_mM$ determines an umbilic direction.  Let $(t_1, t_2)$ be an orthonormal basis of $T_mM$ such that $B\cdot C= 0$ and $|B|\ge |C|.$ Condition (ii)
 is then equivalent to the following three conditions
   $$
 1= u\cdot\alpha(t_1,t_1)= u\cdot b_1,\quad 1= u\cdot\alpha(t_2,t_2)= u\cdot b_2,\quad 0= u\cdot\alpha(t_1,t_2)= u\cdot C.
   $$
Therefore $u\ne 0,\; b_1=\alpha(t_1,t_1)\ne 0$ and $b_2=\alpha(t_2,t_2)\ne 0.$
Also
   $$
\frac12 u \cdot (b_1-b_2)= u\cdot B = 0.
   $$
Since $B$ and $C$ are orthogonal to the non-zero vector $u,$ orthogonal to each other, and $|B|\ge |C|$ we conclude that $C=0.$ Then the curvature ellipse is a
segment and $m$ is semiumbilic. If $b_1$ and $b_2$ where linearly dependent, then both must be equal because their inner products with $u$ are equal. But then $m$
would be umbilic against the hypothesis. If $b_1$ and $b_2$ are linearly independent, then $m$ is not an inflection point, and it is easy to see that
   $$
u=\frac{JB}{H \cdot JB}.
   $$
Conversely, let $m$ be a semiumbilic point that is not an inflection point. Then it is not umbilic. We can choose then the orthonormal basis $(t_1,t_2)$ so that
$C=0.$ It is easy to see that then $H\cdot JB\ne 0,$ so that we can define a vector $u\in N_mM$ by the preceding formula and verify directly that it satisfies
condition (ii).
\end{proof}\hskip 4cm $\square$

\subsection{Application to the asymptotic directions for a surface in $\r5$}

Let $M$ be a surface immersed in $\r 5.$ We denote by $f_{3,u}$ the third order approximation of the function $x\in T_mM \mapsto u\cdot (\exp_m(x)-m),$ that is
  $$
f_{3,u}(x)= u\cdot x + \frac12u\cdot\alpha_m(x,x) -\frac16 \alpha_m(u^\top,x)\cdot\alpha_m(x,x)+\frac16 u^\bot\cdot(\nabla_x\alpha)(x,x).
  $$
In this section, we reword the characterization of asymptotic directions studied in \cite{MRR} and \cite{RF-R-T}.

\begin{definition}
Let $0\ne u\in\r5.$ Then, $u$ determines a \emph{ binormal direction} at $m$ iff the following conditions are true:
\begin{itemize}
\item[(i)] $0$ is a singular point of $f_{3,u};$
\item[(ii)] there is a non-vanishing vector $x\in T_mM$ such that $u\cdot\alpha_m(x,y)=0$ for any $y\in
T_mM$ and such that $f_{3,u}(x)=0.$ We say that such a vector $x$ defines an \it asymptotic direction \rm at $m.$
\end{itemize}
\end{definition}

\begin{theorem}
A vector $0\ne x\in T_mM$ defines an asymptotic direction at $m \in \r5$ if and only if
   \begin{equation}\nonumber
\det\big(\alpha_m(x,t_1),\alpha_m(x,t_2),(\nabla_x\alpha)(x,x)\big)=0.\nonumber
\end{equation}
\end{theorem}

\begin{proof} Assume that $0\ne x\in T_mM$ defines an asymptotic direction. Then there exists $u \in \r5$ with the two properties of the above definition.
These are equi\-valent clearly to the requirements that $u\in N_mM,$ that $u\cdot \alpha_m(x,.)=0$ and that $u\cdot(\nabla_x\alpha)(x,x)=0.$
Now, let $t_1, t_2$ be any basis of $T_mM$. Then the three vectors $\alpha_m(x,t_1),\; \alpha_m(x,t_2),\; (\nabla_x\alpha)(x,x)\in N_mM$ must have a non-vanishing
vector $u\in N_mM$ orthogonal to them all. Since $\dim N_mM=3,$ we conclude that the necessary and sufficient condition for $x$ being an asymptotic direction is that
 those three vectors be linearly dependent, that is
\begin{equation}\nonumber
\det\big(\alpha_m(x,t_1),\alpha_m(x,t_2),(\nabla_x\alpha)(x,x)\big)=0.\nonumber
\end{equation}
\end{proof}
\hskip 1cm $\square$

We have obtained thus a characterization of those asymptotic directions in terms of geometric invariants of the surface. The corresponding equation for the
angle determining those directions can now be computed with the technique used in section 4.1 for the strong principal directions. The program \cite{AMA2}
draws the asymptotic lines, that is those whose tangent is an asymptotic direction at each point.




\section{Appendix}

\verb"Input to be set by user:"
\\
\verb"X = { Cos[u] Cos[v],  Sin[u] Cos[v], Sin[v] };"
\\
\verb"u0 = 1; v0 = 1;"
\\
\\
\verb"Orthonormal tangent basis at {u0, v0}:"
\\
\\
\verb"Xu = D[X, u] /. {u -> u0, v -> v0};  Xu = Xu/Sqrt[Xu.Xu];"
\\
\verb"Xv = D[X, v] /. {u -> u0, v -> v0};"
\\
\verb" Xv = Xv - (Xv.Xu) Xu; Xv = Xv/Sqrt[Xv.Xv];"
\\
\\
\verb"Basis of normal subspace at {u0, v0}:"
\\
\\
\verb"U1 = Cross[Xu, Xv];"
\\
\\
\verb"Verification of both bases: must be non-zero; otherwise,"
\\
\verb"try other initial values for U1, U2, U3 in the above calculation."
\\
\verb" Or perhaps the point is singular."
\\
\\
\verb"N[Det[{Xu, Xv, U1}]]"
\\
\verb"1."
\\
\\
\verb"Direct computation of Monge coordinates"
\\
\\
\verb"Simplify[Solve[X == x Xu  + y Xv +  f1  U1 ]]"
\\
\\
\verb"You may try instead the calculation of the Taylor expansion"
\\
\verb"of Monge coordinates as follows:"
\\
\\
\verb"u = Sum[uu[i, j] x^i y^j, {i, 0, 3}, {j, 0, 3}];"
\\
\verb"v = Sum[vv[i, j] x^i y^j, {i, 0, 3}, {j, 0, 3}];"
\\
\verb"f1 = Sum[g1[i, j] x^i y^j, {i, 0, 3}, {j, 0, 3}];"
\\
\verb"XS = Series[X, {x, 0, 3}, {y, 0, 3}];"
\\
\verb"Monge = LogicalExpand[XS == x Xu  + y Xv +  f1  U1 ] ;"
\\
\verb"Solve[Monge]"
\\
\verb"u =. ; v =. ; f1 =. ;"

 \noindent María García Monera \\Departamento de
Matem\'{a}tica Aplicada \\
Universitat Polit\`{e}cnica de Val\`{e}ncia \\
magar21@upv.es

\vspace*{0.5cm}

\noindent Ángel Montesinos Amilibia \\ Departament de
Geometria i Topologia \\ Universitat de Val\`encia \\
montesin@uv.es

\vspace*{0.5cm}

 \noindent Esther Sanabria Codesal \\ Instituto Universitario
 de Matem\'atica Pura y Aplicada \\ Universitat Polit\`{e}cnica de Val\`{e}ncia \\
esanabri@mat.upv.es

\end{document}